\documentclass{article}
\usepackage{CJKutf8}
\usepackage{amsmath, amsfonts, amscd, bm,  amsthm, amssymb, enumerate, url, float, slashed, MnSymbol}
\usepackage{authblk}
\usepackage{lipsum}
\usepackage[english]{babel}

\usepackage[letterpaper,top=2cm,bottom=2cm,left=3cm,right=3cm,marginparwidth=1.75cm]{geometry}

\usepackage{amsmath}
\usepackage{graphicx}
\usepackage[colorlinks=true, allcolors=blue]{hyperref}

\newcommand\blfootnote[1]{%
\begingroup
\renewcommand\thefootnote{}\footnote{#1}%
\addtocounter{footnote}{-1}%
\endgroup
}

\usepackage{appendix}
\usepackage{tabularx}
\usepackage{xcolor}
\usepackage[numbers,sort&compress]{natbib}
\usepackage{color}

\newtheorem{theorem}{Theorem}[section]
\newtheorem{lemma}[theorem]{Lemma}

\newtheorem{corollary}[theorem]{Corollary}
\newtheorem{remark}[theorem]{Remark}

\numberwithin{equation}{section}
\allowdisplaybreaks[4]

\title{\bf Parabolic frequency monotonicity for two nonlinear equations under Ricci flow}

\author{Chuanhuan Li
\\

\emph{School of Mathematical Sciences, Laboratory of Mathematics and Complex Systems, Beijing Normal University, Beijing 100875, People’s Republic of China}\\

\tt{E-mail: chli@mail.bnu.edu.cn}
\and Yi Li
$^{\ast}$\\
\emph{Shanghai Institute for Mathematics and Interdisciplinary Sciences,
 657 Songhu Road, Yangpu District, Shanghai 200433, China}\\
\tt {E-mail: yilicms@gmail.com} 
\and Kairui Xu
\\
\emph{School of Mathematics, Southeast University, Nanjing 211189, China} \\
\tt{E-mail: xukarry@163.com} 
\and Jichun Zhu
\\
\emph{School of Mathematics, Southeast University, Nanjing 211189, China} \\
\tt{E-mail: 220231969@seu.edu.cn}}

\date{}

\begin{document}
\maketitle

\blfootnote{*Corresponding Author}

\begin{abstract}
In this paper, we study the parabolic frequency for positive solutions of two nonlinear parabolic equations under the Ricci flow on closed manifolds. The first equation is $\partial_{t}u=\Delta_{g(t)}u+au+|\nabla_{g(t)} u|^{2}$ with a constant $a$; the other one is $\partial_{t}u=\Delta_{g(t)} u+\lambda u^{p}$ with two constants $\lambda$ and $p\geq1$. Here $g(t)$ is the Riemannian metric involved by Ricci flow. We establish the monotonicity of the parabolic frequency for the solutions of two nonlinear parabolic equations with bounded Ricci curvature. Subsequently, we apply the parabolic frequency monotonicity to derive some integral type Harnack inequalities. Additionally, we use $-K_{1}$ instead of the lower bound $0$ of Ricci curvature from Theorem 1.3 in \cite{LLX-2023}, where $K_{1}$ is any positive constant. 

${}$

{\bf Keywords} Ricci flow; Parabolic frequency; Nonlinear equation

{\bf 2020 MR Subject Classification} 58J35; 53E20
\end{abstract}

\section{Introduction}

In 1979, the (elliptic) frequency functional for a harmonic function $u(x)$ on $\mathbb{R}^{n}$ was introduced by Almgren in \cite{Dirichlet problem by Almgren}, which was defined by
$$
N(r)=\frac{\displaystyle{r\int_{B(r,p)}|\nabla u(x)|^{2}dx}}{\displaystyle{\int_{\partial B(r,p)}u^{2}(x)\!\ d\sigma}},
$$
where $d\sigma$ is the induced $(n-1)$-dimensional Hausdorff measure on $\partial B(r,p)$, $\nabla$ is the gradient operator on $\mathbb{R}^{n}$, $B(r,p)$ is the ball in $\mathbb{R}^{n}$ and $p$ is a fixed point in $\mathbb{R}^{n}$. Almgren obtained that $N(r)$ is monotone nondecreasing for $r$, and he used this property to investigate the local regularity of harmonic functions and minimal surfaces. Next, Garofalo and Lin \cite{Mon prop by Garofalo, Unique continuation by Garofalo} considered the monotonicity of frequency functional on the Riemannian manifold to study the unique continuation for elliptic operators. The frequency functional was also used to estimate the size of nodal sets in \cite{LA-2018, LA-2018 Yau}. For more applications, see \cite{Harmonic functions by Colding, H-H-L 1998, H-L 1994, frequency by Lin, M-O 2022, Zelditch}.

The parabolic frequency for the solution of the heat equation on $\mathbb{R}^{n}$ was introduced by Poon in \cite{Unique continuation by Poon}, and Ni \cite{N-2015} considered the case when $u(t)$ is a holomorphic function, both of them showed that the parabolic frequency is nondecreasing. Besides, on Riemannian manifolds, the monotonicity of the parabolic frequency was obtained by Colding and Minicozzi \cite{frequency on manifold} through the drift Laplacian operator. Using the matrix Harnack's inequality in \cite{heat equation by Hamilton}, Li and Wang \cite{Parabolic frequency by Li-Wang} investigated the parabolic frequency on compact Riemannian manifolds and the 2-dimensional Ricci flow.
The Ricci flow
\begin{align}
\partial_{t}g(t)=-2\!\ \text{\rm Ric}(g(t))
\label{RF}
\end{align}
was introduced by Hamilton in \cite{RF} to study the compact three-manifolds with positive Ricci curvature, which is a special case of the Poincar\'e conjecture finally proved by Perelman in \cite{Perelman 03, Perelman 02}. Hamilton \cite{RF} obtained the short time existence and uniqueness of the Ricci flow on compact manifolds, and  Shi \cite{Shi-89} obtained a short time solution of the Ricci flow on a complete noncompact manifold whose uniqueness with bounded Riemann curvature was proved by Chen and Zhu in \cite{C-Z 2006}.

In \cite{frequency on RF}, Baldauf and Kim  defined the following parabolic frequency for a solution $u(t)$ of the heat equation 
$$
U(t)=-\frac{\tau(t)\|\nabla_{g(t)}u(t)\|_{L^{2}(d\nu)}^{2}}{\|u(t)\|_{L^{2}(d\nu)}^{2}}\cdot\exp\left\{\displaystyle{-\int_{t_{0}}^{t}\frac{1-\kappa(s)}{\tau(s)}ds}\right\},
$$
where $t\in[t_{0},t_{1}]\subset(0,T)$ with $T<+\infty$, $\tau(t)$ is the backwards time, $\kappa(t)$ is the time-dependent function, $d\nu$ is the weighted measure and $L^{2}(d\nu)$ is the $L^{2}$ space on Riemannian manifold with weighted measure $d\nu$. They proved that parabolic frequency $U(t)$ for the solution of the heat equation is monotone increasing along the Ricci flow with the bounded Bakry-\'{E}mery Ricci curvature and obtained the backward uniqueness. Baldauf, Ho and Lee also derived analogous results to the mean curvature flow in \cite{frequency on MCF}.  

Recently, Liu and Xu studied the monotonicity of parabolic frequency for the weighted $p$-Laplacian heat equation on Riemannian manifolds in \cite{Liu-xu 2022}, and they obtained a theorem of Hardy–\' Polya–Szeg\" o on K\" ahler manifolds under the K\" ahler-Ricci flow.  In \cite{Li-Zhang 2023}, Li and Zhang derived the matrix Li-Yau-Hamilton estimates for positive solutions to the heat equation and the backward conjugate heat equation under the Ricci flow, and they applied these estimates to study the monotonicity of the parabolic frequency. Li, Liu and Ren proved
analogous results under the K\" ahler-Ricci flow in \cite{Li-Liu-Ren 2023}.

In \cite{LLX-2023}, the first author, the second author, and the third author studied the monotonicity of parabolic frequency under the Ricci flow and the Ricci-harmonic flow on closed Riemannian manifolds. They considered two cases: one is the monotonicity of parabolic frequency for the solution of the linear heat equation with bounded Bakry-\' Emery Ricci curvature, and the other case is the monotonicity of parabolic frequency for the solution of the heat equation with bounded Ricci curvature. In \cite{parabolic frequency on G2}, the authors obtained gradient estimates of positive solutions to the heat equation under the $G_{2}$ Laplacian flow for closed $G_{2}$ structures, and they used these estimates to prove the monotonicity of parabolic frequency for positive solutions to the linear heat equation under the $G_{2}$ Laplacian flow with bounded Ricci curvature. Besides, they also got some monotonicity results for the $G_{2}$ Laplacian flow with the bounded Bakry-\'{E}mery Ricci curvature and obtained the backward uniqueness. 

The purpose of this paper is to extend the second case in \cite{LLX-2023} to two nonlinear equations.
At first, we study the following nonlinear equation:
\begin{equation}
\left(\partial_{t}-\Delta_{g(t)}\right)u(t)=a\!\ u(t)+|\nabla_{g(t)}u(t)|_{g(t)}^2,
\label{nonlinear equation}
\end{equation}
where $a$ is any constant independent on $t$, $u(t)$ is a family of smooth functions on $M$, $g(t)$ is the Riemannian metric evolved by the Ricci flow, and $\Delta_{g(t)}$ is the trace Laplacian induced by $g(t)$.

The parabolic frequency $U(t)$ for the solution of nonlinear equation $\eqref{nonlinear equation}$ on Ricci flow $\eqref{RF}$  is denoted by

\begin{equation}
    U(t)=\varphi(t)\frac{D(t)}{I(t)}:=\varphi(t)\frac{\displaystyle{h(t) \int_M\left|\nabla_{g(t)} u(t)\right|_{g(t)}^2 d \mu_{g(t)}}}{\displaystyle{\int_M u(t)^{2} \!\ d\mu_{g(t)}}},
\end{equation}
where
$$\varphi(t)=\exp\left\{-\int_{t_0}^{t}{\left(\frac{h'(s)}{h(s)}+4a+\frac{B(n)}{\eta^{2}e^{as}s}+E(s)\right)ds}\right\},$$
$d\mu_{g(t)}$ is defined in (2.3), $h(t)$ is any time-dependent smooth function and $a$ is any constant, 
$$E(s)=\frac{B_{4}(s)}{\eta e^{as}}+\left(n+nc^{2}A^{2}e^{2as}+\alpha\right)\widetilde{B}(s),\ \ \widetilde{B}(s)=\displaystyle\left(\frac{B_{2}(s)}{\eta^{2}e^{2as}}+B_{3}(s)\right),$$
$c$ is a positive constant depend only on $n$,  $\alpha:= \max\{0,2-\eta e^{at}\}$, $\displaystyle A=\max_{M}u(0)$, $\displaystyle \eta=\min_{M}u(0)$, $B_{2},\ B_{3}$ are  from Lemma \ref{lemma3.4}, and $B(n), B_{4}$ are positive constants from Lemma \ref{lemma3.7}.

By calculating, under the Ricci flow, we obtain the monotonicity of the parabolic frequency for the positive solution of the nonlinear equation $\eqref{nonlinear equation}$.

\begin{theorem}\label{theorem1.1}
    Suppose that $M$ is a closed $n$-dimensional  Riemannian manifold, $(M,g(t))_{t\in [0,T)}$ is the solution of the Ricci flow $\eqref{RF}$ with $-K_{1}g(t)\leq {\rm Ric}(g(t))\leq K_{2}g(t)$ for some $K_{1},K_{2}>0$, and $u(t)$ is a positive solution of the nonlinear equation $\eqref{nonlinear equation}$ with $\eta\leq u(0)\leq A$. Then the following holds:

\begin{itemize} 

\item[(i)] If $h(t)$ is a negative time-dependent function, then the parabolic frequency $U(t)$ is monotone increasing along the Ricci flow.

\item[(ii)] If $h(t)$ is a positive time-dependent function, then the parabolic frequency $U(t)$ is monotone decreasing along the Ricci flow.

\end{itemize}
\end{theorem}

Besides, we also study the following heat-type equation:
\begin{equation}
\left(\partial_{t}-\Delta_{g(t)}\right)u(t)=\lambda \!\ u(t)^{p},
\label{heat-type equation}
\end{equation}
where $p\geq1$, $\lambda\geq 0$ when $p=1$, and for $p>1$, $\lambda$ can be any constant. Both $p$ and $\lambda$ are independent on $t$, $u(t)$ is a family of smooth functions on $M$.

The parabolic frequency $U(t)$ on $[t_{0}, t_{1}]\subset(0,T)$ is denoted by
\begin{equation}
    U(t)=\psi(t)\frac{D(t)}{I(t)}:=\psi(t)\frac{\displaystyle{h(t) \int_M\left|\nabla_{g(t)} u(t)\right|_{g(t)}^2 d \mu_{g(t)}}}{\displaystyle{\int_M u(t)^{2} d\mu_{g(t)}}},
\end{equation}
where
$$\psi(t)=\exp\left\{-\int_{t_0}^{t}{\left(\frac{h^{\prime}(s)+2\lambda_{1} pA^{p-1}h(s)}{h(s)}+\frac{N(s)}{2}n+\frac{ pC(n)}{s}+P\right)ds}\right\},$$
$h(t)$ is any time-dependent smooth function, $\lambda_{1}:=\max\{0,-\lambda\}$,  $\lambda$ is a constant, 
$$
N(t)=C_{2}(t)\left(1+\ln\frac{A}{\eta}\right),\ \ \  \ P=pC_{3}+\lambda_{1}A^{p-1},
$$
$\displaystyle{A=\max_{M\times[t_{0},t_{1}]}u(t)}$,  $\displaystyle{\eta=\min_{M\times[t_{0},t_{1}]}u(t)}$, $C_{2}(t)$ is from Lemma \ref{lemma4.1} and $C(n),\ C_{3}$ are from Lemma \ref{lemma4.2}.

Then using Li-Yau type gradient estimate and Hamilton type gradient estimate on Ricci flow $\eqref{RF}$, we have the following:

\begin{theorem}\label{theorem1.2}
    Suppose that $M$ is a closed $n$-dimensional  Riemannian manifold, $(M,g(t))_{t\in [0,T)}$ is the solution of the Ricci flow $\eqref{RF}$ with $-K_{1}g(t)\leq {\rm Ric}(g(t))\leq K_{2}g(t)$ for some $K_{1},K_{2}>0$, and $u(t)$ is a positive solution of the heat-type equation $\eqref{heat-type equation}$ with $\eta\leq u(t)\leq A$. Then the following holds:

\begin{itemize} 

\item[(i)] If $h(t)$ is a negative time-dependent function, then the parabolic frequency $U(t)$ is monotone increasing along the Ricci flow.

\item[(ii)] If $h(t)$ is a positive time-dependent function, then the parabolic frequency $U(t)$ is monotone decreasing along the Ricci flow.

\end{itemize}
\end{theorem}

\begin{remark}
When $p=1$ in $\eqref{heat-type equation}$, if $\eta\leq u(0)\leq A$, then by the maximum principle, we get $\eta\leq u(t)\leq A$. 
\end{remark}

Using the above parabolic frequency monotonicity, we get the following integral type Harnack inequality.

\begin{corollary}\label{corollary1.3}
    Suppsoe that $M$ is a closed $n$-dimensional  Riemannian manifold, $(M,g(t))_{t\in [0,T)}$ is the solution of the Ricci flow $\eqref{RF}$ with $-K_{1}g(t)\leq {\rm Ric}(g(t))\leq K_{2}g(t)$ for some $K_{1},K_{2}>0$, and $u(t)$ is a positive solution of the heat-type equation $\eqref{heat-type equation}$ with $\eta\leq u(t)\leq A$. Then for any $t\in[t_{0},t_{1}]\subset(0,T)$,
    \begin{equation}
        I(t_1)\geq I(t)\exp\left\{2U(t)\int_{t}^{t_{1}}\frac{-1}{\psi(s)h(s)}ds+2\lambda_{1}\eta^{p-1}(t_{1}-t) \right\}.\label{1.6}
    \end{equation}
\end{corollary}

If we choose $\lambda=0, p=1$ in $\eqref{heat-type equation}$, then we get the 
the parabolic frequency $U(t)$ for the positive solution of the heat equation
\begin{align}
    \label{heat equation}\left(\partial_{t}-\Delta_{g(t)}\right)u(t)=0
\end{align}
on $[t_{0}, t_{1}]\subset(0,T)$ is denoted by
\begin{equation}
    U(t)=\exp\left\{-\int_{t_0}^{t}{\left(\frac{h^{\prime}(s)}{h(s)}+\frac{N(s)}{2}n+\frac{ C(n)}{s}+C_{3}\right)ds}\right\}\frac{D(t)}{I(t)},\label{1.8}
\end{equation}
where $h(t)$ is  any time-dependent smooth function,  $$\displaystyle{N(t)=C_{2}(t)\left(1+\ln\frac{A}{\eta}\right)} ,\ \ \displaystyle A=\max_{M}u(0),\ \  \displaystyle \eta=\min_{M}u(0),$$   $C_{2}(t)$ is from Lemma \ref{lemma4.1} and $C(n),\ C_{3}$ are from Lemma \ref{lemma4.2}. 

From Theorem \ref{theorem1.2}, we get the following corollary:
\begin{corollary}\label{corollary1.4}
    Suppose that $M$ is a closed $n$-dimensional  Riemannian manifold, $(M,g(t))_{t\in [0,T)}$ is the solution of the Ricci flow $\eqref{RF}$ with $-K_{1}g(t)\leq {\rm Ric}(g(t))\leq K_{2}g(t)$ for some $K_{1},K_{2}>0$, and $u(t)$ is a positive solution of the heat equation $\eqref{heat equation}$ with $\eta\leq u(0)\leq A$. Then the following holds:

\begin{itemize}

\item[(i)] If $h(t)$ is a negative time-dependent function, then the parabolic frequency $U(t)$ is monotone increasing along the Ricci flow.

\item[(ii)] If $h(t)$ is a positive time-dependent function, then the parabolic frequency $U(t)$ is monotone decreasing along the Ricci flow.

\end{itemize}
\end{corollary}

\begin{remark}\label{remark1.5}
     Since we use Hamilton-Souplet-Zhang-type gradient estimates instead of Hamilton's estimate and Li-Yau's gradient estimates with $-K_{1}g(t)\leq {\rm Ric}(g(t))\leq K_{2}g(t)$ for some $K_{1}, K_{2}>0$, the parabolic frequency $\eqref{heat equation}$ is different from Theorem 4.3 in \cite{LLX-2023}.
\end{remark}

${}$

This paper is organized as follows. In Section 2, we give some notations and definitions. In Section 3, we consider the monotonicity of the parabolic frequency for the positive solution to the nonlinear equation $\eqref{nonlinear equation}$ under the Ricci flow $\eqref{RF}$ and give the proof of Theorem 1.1. In Section 4, we consider the monotonicity of parabolic frequency for the positive solution to the heat-type equation $\eqref{heat-type equation}$ under the Ricci flow $\eqref{RF}$ and give the proof of Theorem 1.2 and Corollary 1.4.

\section{notations and definitions}\label{section 2}

In this section, we introduce some notations and definitions which will be used in the sequel. We use the notations in Hamilton’s paper \cite{RF}, $\nabla_{g}$ is the Levi-Civita connection induced by $g$, $\text{\rm Ric}$, $R$, $dV_{g}$ are  Ricci curvature, scalar curvature, and volume form, respectively. The Laplacian of the smooth time-dependent function $f(t)$ with respect to a family of Riemannian metrics $g(t)$ is
$$
\Delta_{g(t)}f(t)=g^{ij}(t)\left[\partial_{i}\partial_{j}f(t)-\Gamma_{ij}^{k}(t)\partial_{k}f(t)\right],
$$
where $\Gamma_{ij}^{k}(t)$ is the Christoffel symbol of $g(t)$ and $\displaystyle{\partial_{i}=\frac{\partial}{\partial x^{i}}}$.

Under the Ricci flow $\eqref{RF}$, let  $\tau(t)=T-t$ be the backward time. For any time-dependent smooth function $f(t)$ on $M$, we denote 
$$K(t)=(4\pi\tau(t))^{-\frac{n}{2}}e^{-f(t)}$$
to be the positive solution of the conjugate heat equation
\begin{equation}
\partial_{t}K(t)=-\Delta_{g(t)} K(t)+R(g(t)) K(t).\label{2.1}
\end{equation}
From the definition of $K(t)$, we can prove the smooth function $f(t)$ satisfies the following equation
\begin{equation}
\partial_{t}f(t)=-\Delta_{g(t)} f(t)-R(g(t))+|\nabla_{g(t)} f(t)|_{g(t)}^{2}+\frac{n}{2\tau(t)}.\label{2.2}
\end{equation}
Then, under the Ricci flow $\eqref{RF}$, we can define the weighted measure 
\begin{equation}
d\mu_{g(t)}:= K(t) dV_{g(t)}=(4\pi\tau(t))^{-\frac{n}{2}}e^{-f(t)}dV_{g(t)},\ \ \int_{M} d\mu_{g(t)}=1.\label{2.3}
\end{equation}

On the weighted Riemannian manifold $(M^{n},g(t),d\mu_{g(t)})$, the weighted Bochner formula for any smooth function $u$ is as follow 
\begin{align}
\label{Bochner}
\Delta_{g(t),f(t)}\left(|\nabla_{g(t)} u|_{g(t)}^{2}\right)=&2|\nabla_{g(t)}^{2} u|_{g(t)}^{2}+2\left\langle\nabla_{g(t)} u,\nabla_{g(t)}\Delta_{g(t),f(t)}u\right\rangle_{g(t)}\\
&+2\text{\rm Ric}_{f(t)}\left(\nabla_{g(t)} u,\nabla_{g(t)} u\right)\notag,
\end{align}
where
\begin{align}
\text{\rm Ric}_{f(t)}:=\text{\rm Ric}(g(t))+\nabla_{g(t)}^{2} f(t)
\label{Baker}
\end{align}
is the Bakry-\'{E}mery Ricci tensor introduced in \cite{B-E Ricci}, and
\begin{equation}
\Delta_{g(t),f(t)}u:=e^{f(t)}\text{div}_{g(t)}\left(e^{-f(t)}\nabla_{g(t)}{u}\right)=\Delta_{g(t)} u-\left\langle\nabla_{g(t)} f(t),\nabla_{g(t)} u\right\rangle_{g(t)}\label{2.6}
\end{equation}
is the drift Laplacian operator for any smooth function $u$.
And we denote $W^{2,2}_{0}(d\mu_{g(t)})$ as the Sobolev space on Riemannian manifold $(M^{n},g(t))$ with weighted measure $d\mu_{g(t)}$, and the subscript $0$ means 
$\nabla^{\alpha}u=0$ on the boundary for $|\alpha|\leq 1$.

Under the Ricci flow $\eqref{RF}$, the volume form $dV_{g(t)}$ satisfies
$$\partial_{t}(dV_{g(t)})=-R(g(t))\!\ dV_{g(t)}.
$$
Thus, the conjugate heat kernel measure $d\mu_{g(t)}$ is evolved by
\begin{equation}
\partial_{t}(d\mu_{g(t)})=-\left(\Delta_{g(t)} K(t)\right)dV_{g(t)}=-\frac{\Delta_{g(t)} K(t)}{K(t)}d\mu_{g(t)}.
\label{volume}
\end{equation}


For convenient, we use $\Delta_{f},\ \text{\rm Ric},\ \nabla,\ |\cdot|,\ d\mu$ to replace $\Delta_{g(t),f(t)},\ \text{\rm Ric}(g(t)),$ $\nabla_{g(t)},\ |\cdot|_{g(t)},\ d\mu_{g(t)}$. We always omit
the time variable $t$.

\section{parabolic frequency of nonlinear equation}

In this section, we consider the parabolic frequency $U(t)$ for the positive solution of the nonlinear equation $\eqref{nonlinear equation}$ under the Ricci flow $\eqref{RF}$.

For a time-dependent function $u=u(t):M\times [t_0,t_1]\rightarrow \mathbb{R}^+$ with $u(t),\partial_{t}u(t)\in W^{2,2}_{0}(d\mu_{g(t)})$ for all $t\in [t_0,t_1] \subset (0,T)$, we define
\begin{align}
I({t}) & =\int_M u(t)^2 d \mu_{g(t)},\\
D(t) & =h(t) \int_M\left|\nabla_{g(t)} u(t)\right|_{g(t)}^2 d \mu_{g(t)}\\
& =-h(t) \int_M\left<u(t), \Delta_{g(t), f(t)} u(t)\right>_{g(t)} d \mu_{g(t)}\notag,\\
    U(t)&=\exp\left\{-\int_{t_0}^{t}{\left(\frac{h'(s)}{h(s)}+4a+\frac{B(n)}{\eta^{2}e^{as}s}+E(s)\right)ds}\right\}\frac{D(t)}{I(t)},
\end{align}
where
$h(t)$ is any time-dependent smooth function and $a$ is any constant, 
$$E(s)=\frac{B_{4}(s)}{\eta e^{as}}+\left(n+nc^{2}A^{2}e^{2as}+\alpha\right)\widetilde{B}(s),\ \ \widetilde{B}(s)=\displaystyle\left(\frac{B_{2}(s)}{\eta^{2}e^{2as}}+B_{3}(s)\right),$$
$c$ is the positive constant depend only on $n$,  $\alpha:= \max\{0,2-\eta e^{at}\}$, $\displaystyle A=\max_{M}u(0)$, $\displaystyle \eta=\min_{M}u(0)$, $B_{2},\ B_{3}$ are constants from Lemma \ref{lemma3.4}, and $B(n), B_{4}$ are positive constants from Lemma \ref{lemma3.7}.

\begin{lemma}\label{lemma3.1}
Suppose $u(t)$ is a family of smooth functions. Under the Ricci flow $\eqref{RF}$, the norm of $|\nabla_{g(t)}u(t)|_{g(t)}^2$ satisfies the following evolution equation:
\begin{align}
    (\partial_t-\Delta_{g(t)})|\nabla_{g(t)}u(t)|_{g(t)}^2&=-2|\nabla_{g(t)}^2u(t)|_{g(t)}^2+2\left<\nabla_{g(t)}u(t),\nabla_{g(t)}(\partial_t-\Delta_{g(t)})u(t)\right>_{g(t)}.
\end{align}
\end{lemma}
 \begin{proof}
     This lemma is from Lemma 3.1 in  \cite{LLX-2023}.
 \end{proof}

If we let $u=\ln v$, then the nonlinear equation $\eqref{nonlinear equation}$ can be viewed as the following nonlinear equation:
\begin{equation}
    \left(\partial_{t}-\Delta_{g(t)}\right)v(t)=a \!\ v(t)\ln v(t).
    \label{nonlinear equation2}
\end{equation}
 Hence we need to give the following estimates:

\begin{lemma}\label{lemma3.2}
    Suppose that $M$ is a closed $n$-dimensional  Riemannian manifold, $(M,g(t))_{t\in [0,T)}$ is the solution of the Ricci flow $\eqref{RF}$ with $-K_{1}g(t)\leq {\rm Ric}(g(t))\leq K_{2}g(t)$ for some $K_{1},K_{2}>0$, and $v(t)$ is a positive solution of the nonlinear equation $\eqref{nonlinear equation2}$ with $\eta_{1}\leq v(0)\leq A_{1}$. Then for all $t\in[t_{0},t_{1}]\subset(0,T)$, we have
    $$
    \frac{|\nabla_{g(t)} v(t)|_{g(t)}}{v(t)}\leq B_{1}\left(\sqrt{\overline{K}}+\frac{1}{R(t)}+\frac{1}{\sqrt{T}}+\sqrt{X}\right)\left(1+\ln\frac{A_{1}^{e^{|a|T}}}{v(t)}\right),
    $$
   where
   $$
   B_{1}=B_{1}(n), \ \overline{K}=\max\{K_{1},K_{2}\},\ 0<R(t)<\rho(t),\ \rho(t)={\rm diam}(M,g(t)), \ $$
    and $X=\max\left\{a\left(1+e^{|a|T}\ln A_{1}\right),0\right\}$ are all positive constants. 
\end{lemma}
\begin{proof}
Applying the maximum principle in \cite{CK 2004} to the nonlinear equation $\eqref{nonlinear equation2}$ with $\eta_{1}\leq v(0)\leq A_{1}$, we get
    $$
    \eta_{1}^{e^{at}}\leq v(t)\leq A_{1}^{e^{at}}{\leq A_{1}^{e^{|a|T}}}.
    $$
   Since Riemannian manifold $M$ is closed, $M$ can be covered by finite balls for each $t\in [t_{0},t_{1}]$, which means for any $p_{i}\in M$,
    $$\displaystyle\bigcup_{i=1}^{n}B_{p_{i}}(R(t))=(M,g(t)).$$
    Let $b=0$ in Theorem 1 of \cite{Hamilton for heat-type equation}, we get the desired result.
\end{proof}

\begin{remark}\label{remark3.3}
    When we assume the solution $u(t)$ of nonlinear equation $\eqref{nonlinear equation}$ is positive, the solution $v(t)$ of nonlinear equation $\eqref{nonlinear equation2}$ should be more than 1. Hence we need  $\eta_{1}>1$. 
\end{remark}

\begin{lemma}\label{lemma3.4}
    Suppose that $M$ is a closed $n$-dimensional Riemannian manifold, $(M,g(t))_{t\in [0,T)}$ is the solution of the Ricci flow $\eqref{RF}$ with $-K_{1}g(t)\leq {\rm Ric}(g(t))\leq K_{2}g(t)$ for some $K_{1},K_{2}>0$, and $u(t)$ is a positive solution of the nonlinear equation $\eqref{nonlinear equation}$ with $\ln\eta_{1}=\eta\leq u(0)\leq A=\ln A_{1}$. Then for all $t\in[t_{0},t_{1}]\subset(0,T)$, we have
    \begin{eqnarray*}
    |\nabla_{g(t)} u(t)|^{2}_{g(t)}&\leq &B_{1}^{2}\left(\sqrt{\overline{K}}+\frac{1}{{R(t)}}+\frac{1}{\sqrt{T}}+\sqrt{X}\right)^{2}\left(1+e^{{|a|T}}A-u\right)^{2}\notag\\
    &\leq& B_{2}(t)+B_{3}(t)u^{2},\notag
    \end{eqnarray*}
    where $B_{1}$, $\overline{K}$, $R(t)$, and $X$ are from Lemma \ref{lemma3.2}.
    $$B_{2}(t)=B_{3}(t)\left(1+e^{|a|T}A\right)^{2},\quad B_{3}=2B_{1}^{2}\left(\sqrt{\overline{K}}+\frac{1}{{R(t)}}+\frac{1}{\sqrt{T}}+\sqrt{X}\right)^{2}.$$
\end{lemma}

\begin{lemma}\label{lemma3.5}
Suppose that $M$ is a closed $n$-dimensional  Riemannian manifold, $(M,g(t))_{t\in [0,T)}$ is the solution of the Ricci flow $\eqref{RF}$ with $-K_{1}g(t)\leq {\rm Ric}(g(t))\leq K_{2}g(t)$ for some $K_{1},K_{2}>0$, and $v(t)$ is a positive solution of the nonlinear equation $\eqref{nonlinear equation2}$. Then for all $t\in[t_{0},t_{1}]\subset(0,T)$, we have 
\begin{equation}
\frac{\left|\nabla_{g(t)} v(t)\right|^2_{g(t)}}{v(t)^{2}}-2 \frac{\partial_{t}v(t)}{v(t)}-2 a \ln v(t)\leq\frac{B(n)}{t}+B_{4}(n,\overline{K},R(t),a),\label{3.6}
\end{equation}
where $B(n)$ is a positive constant  depends only on $n$,   
$$B_{4}(n,\overline{K}, R(t),a):=\frac{B(n)}{{R(t)}}^{2}+\frac{B(n)\sqrt{\overline{K}}}{{R(t)}}+B(n)\overline{K}+16n\sqrt{a^{2}},$$
 and $\overline{K}=\max\{K_{1},K_{2}\}$, $0<R(t)\leq \rho(t)$, $\rho(t)={\rm diam}(M,g(t))$.
\end{lemma}
\begin{proof}
    Similar to Lemma \ref{lemma3.2}, the closed Riemannian manifold $M$ can be covered by finite balls. If we let $\displaystyle{c=d=\frac{1}{4}}$, $\beta=2$, $\displaystyle{\epsilon=\frac{1}{2}}$, $q=0$, $\gamma=0$ in Theorem 3.2 of \cite{Harnack estimates for nonlinear equation}, then we get the desired result.
\end{proof}


\begin{lemma}\label{lemma3.7}
    Suppose that $M$ is a closed $n$-dimensional  Riemannian manifold, $(M,g(t))_{t\in [0,T)}$ is the solution of the Ricci flow $\eqref{RF}$ with $-K_{1}g(t)\leq {\rm Ric}(g(t))\leq K_{2}g(t)$ for some $K_{1},K_{2}>0$, and $u(t)$ is a positive solution of the nonlinear equation $\eqref{nonlinear equation}$. Then for all $t\in[t_{0},t_{1}]\subset(0,T)$, we have 
    \begin{align}
    |\nabla_{g(t)} u(t)|^{2}_{g(t)}-2\partial_{t}u(t)-2 au(t)\leq \frac{B(n)}{t}+B_{4}(n,\overline{K},R(t),a).
    \end{align}
    where $B(n)$ and $B_{4}(n,\overline{K},R(t),a)$ are from Lemma \ref{lemma3.5}.
\end{lemma}


\begin{theorem}\label{theorem3.8}
Suppose that $M$ is a closed $n$-dimensional  Riemannian manifold, $(M,g(t))_{t\in [0,T)}$ is the solution of the Ricci flow $\eqref{RF}$ with bounded Ricci curvature, $-K_{1}g(t)\leq \mathrm{Ric}(g(t)) \leq K_{2}g(t)$, where $K_{1},K_{2}$ are positive constants, and $u(t)$ is a positive solution of the nonlinear equation $\eqref{nonlinear equation}$ with $\eta\leq u(0)\leq A$. Then the following holds:

\begin{itemize} 

\item[(i)] If $h(t)$ is a negative time-dependent function, then the parabolic frequency $U(t)$ is monotone increasing along the Ricci flow.

\item[(ii)] If $h(t)$ is a positive time-dependent function, then the parabolic frequency $U(t)$ is monotone decreasing along the Ricci flow.
\end{itemize}
\end{theorem}
\begin{proof}
Before discussing the monotonicity of $U(t)$, we need to calculate the derivative of $I(t)$ and $D(t)$. From the proof of Lemma \ref{lemma3.2}, we get an estimate of $v(t)$, recall that $u(t)=\ln v(t)$, which implies 
$$\eta e^{at}\leq u(t)\leq A e^{at},$$
then together with Lemma \ref{lemma3.4} yields
\begin{align}
|\nabla u|^{2}\leq \left(\frac{B_{2}}{u^{2}}+B_{3}\right)u^{2}\leq \left(\frac{B_{2}}{\eta^{2}e^{2at}}+B_{3}\right)u^{2}.\label{3.8}
\end{align}
Therefore, by Lemma \ref{lemma3.7} and Young's inequality, we have
    \begin{align}
        I^{\prime}(t)&=\int_{M}{\left(2u\cdot u_{t}-u^{2}\frac{\Delta K}{K}\right)}d\mu\notag\\
            &=\int_{M}{\left(2u\cdot u_{t}-\Delta(u^2)\right)}d\mu\notag\\
            &=\int_{M}\left(2u\cdot u_{t}-2u\Delta u-2|\nabla u|^2\right)d\mu,\notag\\
            &=2\int_{M}{u\left(u_{t}-\frac{\eta e^{at}|\nabla u|^2}{2u}\right)}d\mu-2\int_M{u\Delta u}d\mu-(2-\eta e^{at})\int_M{|\nabla u|^2}d\mu\notag\\
            &\geq2\int_{M}{u\left(u_{t}-\frac{|\nabla u|^2}{2}\right)}d\mu-2\int_M{u\Delta u}\!\ d\mu-\left(2-\eta e^{at}\right)\int_M{|\nabla u|^2}d\mu\notag\\
            &\geq \int_{M}{u\left(2u_{t}-|\nabla u|^2\right)}d\mu-\frac{1}{\epsilon (t)}\int_{M}{u^2}d\mu-\epsilon (t)\int_{M}{|\Delta u|^2}d\mu\notag\\
            &\quad-\left(2-\eta e^{at}\right)\int_{M}{|\nabla u|^2}d\mu\notag\\
            & \geq -\left(2a+\frac{B(n)}{\eta e^{at} t}+\frac{ B_{4}}{\eta e^{at}}+\frac{1}{\epsilon (t)}+\alpha \widetilde{B}(t)\right)\int_{M}{u^2}d\mu-\epsilon (t)\int_{M}{|\Delta u|^2}d\mu\notag,
    \end{align}
    where $\epsilon(t)$ is a constant depend only on time $t$, $\alpha:= \max\{0,2-\eta e^{at}\}$, and $\widetilde{B}(t)=\displaystyle\left(\frac{B_{2}}{\eta^{2}e^{2at}}+B_{3}\right)$. 
    
    For the derivative of $D(t)$, from Lemma \ref{lemma3.1}, we have
    \begin{align}
             D^{\prime}(t)&=\frac{d}{dt}\left(h\int_{M}{|\nabla u|^{2}}d\mu\right)\\
            &=h^{\prime}\int_{M}{|\nabla u|^2}d\mu+h\int_{M}{(\partial_t-\Delta)|\nabla u|^2}d\mu\notag\\
            &=h^{\prime}\int_{M}{|\nabla u|^2}d\mu+h\int_{M}\left({2\left<\nabla u,\nabla(\partial_t-\Delta)u\right>-2|\nabla^2 u|^2}\right)d\mu\notag\\
            &=h^{\prime}\int_{M}{|\nabla u|^2}d\mu+h\int_{M}\left({2\left<\nabla u,\nabla\left(au+|\nabla u|^2\right)\right>-2|\nabla^2 u|^2}\right)d\mu\notag\\
             &=(h^{\prime}+2ah)\int_{M}{|\nabla u|^2}d\mu-2h\int_{M}{|\nabla^2 u|^2}d\mu+2h\int_{M}\langle\nabla u,\nabla|\nabla u|^{2}\rangle d\mu\notag,
            \end{align}  
        If $h(t)< 0$, applying Lemma \ref{lemma3.4} and the following inequalities
        \begin{align}
        \langle\nabla u,\nabla|\nabla u|^{2}\rangle&=2\nabla_{i}\nabla_{j}u\nabla_{i}u\nabla_{j}u
        \leq c(n)|\nabla u|^{2}|\nabla^{2}u|\notag,\\
        2|\nabla^{2} u|\cdot|\nabla u|^2&\leq \frac{1}{n c(n)} 
            |\nabla^{2} u|^2 +n c(n)|\nabla u|^4,\notag\\
        |\nabla^{2}u|^{2}&\geq\frac{1}{n}|\Delta u|^{2},\notag
        \end{align}
          where $c(n)$ is the positive constant depend only on $n$, then we get
            \begin{align}
                D^{\prime}(t)&\geq (h^{\prime}+2ah)\int_{M}{|\nabla u|^2}d\mu-\frac{1}{n}\cdot h\int_{M}{|\Delta u|^2}d\mu+nhc(n)^{2}\int_{M}|\nabla u|^{4}d\mu\notag\\
                &\geq \bigg[h^{\prime}+h\left(2a+nc^{2}\widetilde{B}(t)A^{2}e^{2at}\right)\bigg]\int_{M}{|\nabla u|^2}d\mu-\frac{1}{n}\cdot h\int_{M}{|\Delta u|^2}d\mu\notag.
            \end{align}
        Together with estimates of $I^{\prime}(t)$ and $D^{\prime}(t)$, from $\eqref{3.8}$, if we let $\displaystyle{\epsilon(t)=\frac{1}{n\widetilde{B}(t)}}$,  then the parabolic frequency $U(t)$ satisfies
        \begin{align}
            I^{2}(t)U^{\prime}(t)&\geq\exp\left\{-\int_{t_0}^{t}{\left(\frac{h'(s)}{h(s)}+4a+\frac{B(n)}{\eta^{2}e^{as}s}+E(s)\right)ds}\right\}\notag\\
            &\quad\cdot(-h)\cdot\left(\frac{1}{n}I(t)\int_{M}|\Delta u|^{2}d\mu-\frac{1}{n \widetilde{B}(t)}\int_{M}|\Delta u|^{2}d\mu\int_{M}|\nabla u|^{2}d\mu\right)\notag\\
            &\geq0.\notag
        \end{align}
        where $\displaystyle E(s)=\left(n+nc^{2}A^{2}e^{2as}+\alpha\right)\widetilde{B}(s)+\frac{B_{4}(s)}{\eta e^{as}}$.

        In the case $h(t)> 0$, we have 
        \begin{align}
                D^{\prime}(t)&\leq (h^{\prime}+2ah)\int_{M}{|\nabla u|^2}d\mu-\frac{1}{n}\cdot h\int_{M}{|\Delta u|^2}d\mu+nh\int_{M}|\nabla u|^{4}d\mu\\
                &\leq \bigg[h^{\prime}+h\left(2a+nc^{2}\widetilde{B}(t)A^{2}e^{2at}\right)\bigg]\int_{M}{|\nabla u|^2}d\mu-\frac{1}{n}\cdot h\int_{M}{|\Delta u|^2}d\mu\notag.
            \end{align}
        Thus, we obtain
        \begin{align}
            I^{2}(t)U^{\prime}(t)&\leq\exp\left\{-\int_{t_0}^{t}{\left(\frac{h'(s)}{h(s)}+4a+\frac{B(n)}{\eta^{2}e^{as}s}+E(s)\right)ds}\right\}\notag\\
            &\quad\cdot(-h)\cdot\left(\frac{1}{n}I(t)\int_{M}|\Delta u|^{2}d\mu-\frac{1}{n \widetilde{B}(t)}\int_{M}|\Delta u|^{2}d\mu\int_{M}|\nabla u|^{2}d\mu\right)\notag\\
            &\leq0.\notag
        \end{align}
  Together with the above, we get the desired results. 
\end{proof}

If we let
$$\varphi(t)=\exp\left\{-\int_{t_0}^{t}{\left(\frac{h'(s)}{h(s)}+4a+\frac{B(n)}{\eta^{2}e^{as}s}+E(s)\right)ds}\right\},$$
then we get the following Corollary:

\begin{corollary}\label{corollary3.9}
Suppose that $M$ is a closed $n$-dimensional  Riemannian manifold, $(M,g(t))_{t\in [0,T)}$ is the solution of the Ricci flow $\eqref{RF}$ with bounded Ricci curvature, $-K_{1}g(t)\leq \mathrm{Ric}(g(t)) \leq K_{2}g(t)$, where $K_{1},K_{2}$ are positive constants, and $u(t)$ is a positive solution of the nonlinear equation $\eqref{nonlinear equation}$ with $\eta\leq u(0)\leq A$. Then for any $t\in[t_{0},t_{1}]\subset(0,T)$,

\begin{itemize} 

\item[(i)] If $0<\eta e^{at}\leq 1$, then
   \begin{equation}
        I(t_1)\geq I(t)\exp\left\{2U(t)\int_{t}^{t_{1}}\frac{\eta e^{as}-1}{\varphi(s)h(s)}ds+2a(t_{1}-t) \right\}.\notag
    \end{equation}

\item[(ii)] If $\eta e^{at}>1$, then
   \begin{equation}
        I(t_1)\geq I(t)\exp\left\{2U(t_{1})\int_{t}^{t_{1}}\frac{\eta e^{as}-1}{\varphi(s)h(s)}ds+2a(t_{1}-t) \right\}.\notag
    \end{equation}
\end{itemize}
\end{corollary}

    \begin{proof}
       We give the proof of case $h(t) < 0$ (The case $h(t) > 0$ is similar to it). According to the definition of $U(t)$ and $\eta e^{at}\leq u(t)\leq Ae^{at}$, yields
        \begin{equation}\label{3.11}
            \frac{d}{dt}\ln I(t)=\frac{I^{\prime}(t)}{I(t)} \geq \frac{2\eta e^{at}-2}{h(t)}\cdot\frac{D(t)}{I(t)}+2a=\frac{2\eta e^{at}-2}{\varphi(t)h(t)}U(t)+2a,
        \end{equation}
   By Theorem \ref{theorem3.8}, integrate $\eqref{3.11}$ from $t$ to $t_{1}$, we get
        \begin{align}
            \ln I(t_1)-\ln I(t)&\geq \int_{t}^{t_{1}}\left(\frac{2\eta e^{as}-2}{\varphi(s)h(s)}U(s)+2a\right) ds\notag.
        \end{align}
        In the case $0<\eta e^{at}\leq1$, which implies $\eta e^{at}-1\leq0$, thus
        \begin{align}
            \ln I(t_1)-\ln I(t)\geq 2U(t)\int_{t}^{t_{1}}\frac{\eta e^{as}-1}{\varphi(s)h(s)}ds+2a(t_{1}-t)\notag
        \end{align}
        From the boundedness of time-dependent function $h(t)$ and $\varphi(t)$, we have
    \begin{equation}
        I(t_1)\geq I(t)\exp\left\{2U(t)\int_{t}^{t_{1}}\frac{\eta e^{as}-1}{\varphi(s)h(s)}ds+2a(t_{1}-t) \right\}.
    \end{equation}
    If $\eta e^{at}>1$, then $\eta e^{at}-1\geq0$. Similarly, we get
        \begin{equation}
        I(t_1)\geq I(t)\exp\left\{2U(t_{1})\int_{t}^{t_{1}}\frac{\eta e^{as}-1}{\varphi(s)h(s)}ds+2a(t_{1}-t) \right\}.
    \end{equation}
    Then we get the desired result.
    \end{proof}

\section{parabolic frequency of heat-type equation}

In this section, we consider the parabolic $U(t)$ for the positive solution of the heat-type equation $\eqref{heat-type equation}$ under the Ricci flow $\eqref{RF}$.

For a time-dependent function $u=u(t):M\times [t_0,t_1]\rightarrow \mathbb{R}^+$ with $u(t),\partial_{t}u(t)\in W^{2,2}_{0}(d\mu_{g(t)})$ for all $t\in [t_0,t_1] \subset (0,T)$, we define
\begin{align}
I({t}) & =\int_M u(t)^2 d \mu_{g(t)},\\
D(t) & =h(t) \int_M\left|\nabla_{g(t)} u(t)\right|_{g(t)}^2 d \mu_{g(t)}\\
& =-h(t) \int_M\left<u(t), \Delta_{g(t), f(t)} u(t)\right>_{g(t)} d \mu_{g(t)}\notag,
\end{align}
\begin{equation}
    U(t)=\exp\left\{-\int_{t_0}^{t}{\left(\frac{h^{\prime}(s)+2\lambda_{1} pA^{p-1}h(s)}{h(s)}+\frac{N(s)}{2}n+\frac{ pC(n)}{s}+P\right)ds}\right\}\frac{D(t)}{I(t)},
\end{equation}
where
$\lambda_{1}:=\max\{0,-\lambda\}$, $h(t)$ is any time-dependent smooth function,  $\lambda$ is a constant, 
$$
N(t)=C_{2}(t)\left(1+\ln\frac{A}{\eta}\right), \ P=pC_{3}+\lambda_{1}A^{p-1}, \ A=\max_{M\times[t_{0},t_{1}]}u(t),  \eta=\min_{M\times[t_{0},t_{1}]}u(t), 
$$
$C_{2}(t)$ is  from Lemma \ref{lemma4.1} and $C(n),\ C_{3}$ are constants from Lemma \ref{lemma4.2}.

\begin{lemma}\label{lemma4.1}
    Suppose that $M$ is a closed $n$-dimensional  Riemannian manifold, $(M,g(t))_{t\in [0,T)}$ is the solution of the Ricci flow $\eqref{RF}$ with $-K_{1}g(t)\leq {\rm Ric}(g(t))\leq K_{2}g(t)$ for some $K_{1},K_{2}>0$, and $u(t)$ is a positive solution of the heat-type equation $\eqref{heat-type equation}$ with $\eta\leq u(t)\leq A$. Then for all $t\in[t_{0},t_{1}]\subset(0,T)$, we have
    $$
    \frac{|\nabla_{g(t)} u(t)|_{g(t)}}{u(t)}\leq C_{2}(t)\left(1+\ln\frac{A}{u(t)}\right),$$
    where  $$C_{2}(t)=C_{1}\left(\frac{1}{R(t)}+\frac{1}{\sqrt{T}}+\sqrt{\overline{K}}+\sqrt{\alpha}\right),$$
    $C_{1}$ is a positive constant depend only on $n$, $\overline{K}=\max\{K_{1},K_{2}\}$, $0<R(t)<\rho(t)$, $\rho(t)={\rm diam} (M,g(t))$ and $\alpha=\max\{p\lambda A^{p-1},0\}$.
\end{lemma}
\begin{proof}
    Since Riemannian manifold $M$ is closed, $M$ can be covered by finite balls for each $t\in [t_{0},t_{1}]$, which means for any $p_{i}\in M$,
    $$\displaystyle\bigcup_{i=1}^{n}B_{p_{i}}(R(t))=(M,g(t)).$$
    Then from Theorem 2 in \cite{Hamilton for heat-type equation}, we prove this Lemma.
\end{proof}

\begin{lemma}\label{lemma4.2}
    Suppose that $M$ is a closed $n$-dimensional  Riemannian manifold, $(M,g(t))_{t\in [0,T)}$ is the solution of the Ricci flow $\eqref{RF}$ with $-K_{1}g(t)\leq {\rm Ric}(g(t))\leq K_{2}g(t)$ for some $K_{1},K_{2}>0$, and $u(t)$ is a positive solution of the heat-type equation $\eqref{heat-type equation}$ with $\eta\leq u(t)\leq A$. Then for all $t\in[t_{0},t_{1}]\subset(0,T)$, we have
    \begin{align}
    \frac{|\nabla_{g(t)} u(t)|^{2}_{g(t)}}{u(t)^{2}}+\frac{\lambda}{p}u(t)^{p-1}-\frac{1}{p}\frac{\partial_{t}u(t)}{u(t)}&\leq\frac{C(n)}{t}+C(n)\left(1+K_{1}+\overline{K}\right)\notag\\
    &\quad+C(n)p^{2}\lambda A^{p-1}\notag\\
    &=\frac{C(n)}{t}+C_{3}\left(\overline{K},\lambda,p,n,A\right),\notag
    \end{align}
    where $\overline{K}=\max\{K_{1},K_{2}\}$ and $C(n)$ is a positive constant depend only on $n$.
\end{lemma}
\begin{proof}
    This Lemma is from Theorem 1.2 in \cite{Harnack estimates for heat-type equation}, here we let $k=2p$, $\theta=0$, $\gamma=0$ and $h\equiv\lambda$.
\end{proof}

\begin{theorem}\label{theorem4.3}
    Suppose that $M$ is a closed $n$-dimensional  Riemannian manifold, $(M,g(t))_{t\in [0,T)}$ is the solution of the Ricci flow $\eqref{RF}$ with $-K_{1}g(t)\leq {\rm Ric}(g(t))\leq K_{2}g(t)$ for some $K_{1},K_{2}>0$, and $u(t)$ is a positive solution of the heat-type equation $\eqref{heat-type equation}$ with $\eta\leq u(t)\leq A$. Then the following holds:

\begin{itemize} 

\item[(i)] If $h(t)$ is a negative time-dependent function, then the parabolic frequency $U(t)$ is monotone increasing along the Ricci flow.

\item[(ii)] If $h(t)$ is a positive time-dependent function, then the parabolic frequency $U(t)$ is monotone decreasing along the Ricci flow.
\end{itemize}
\end{theorem}
\begin{proof}
    Before discussing the monotonicity of $U(t)$, we need to calculate the derivative of $I(t)$ and $D(t)$. Using Young’s identity and Lemma \ref{lemma4.2}, we get the derivative of $I(t)$.
    \begin{align}
        I^{\prime}(t)&=\frac{d}{dt}\left(\int_{M}u^{2}d\mu\right)=\int_{M}\left(2u\partial_{t}u-\Delta u^{2}\right)d\mu\notag\\
        &=2\int_{M}\left[\left(u\cdot\partial_{t}u-p|\nabla u|^{2}\right)+(p-1)|\nabla u|^{2}\right]d\mu-2\int_{M}u\Delta u \!\ d\mu\notag\\
        &\geq-2\left(\lambda_{1} A^{p-1}+\frac{pC(n)}{t}+pC_{3}\right)I(t)-2\int_{M}u\Delta u \!\ d\mu+2(p-1)\int_{M}|\nabla u|^{2}d\mu\notag\\
        &\geq-2\left(\lambda_{1} A^{p-1}+\frac{pC(n)}{t}+pC_{3}+\frac{N(t)}{2}n\right)I(t)-\frac{2}{N(t)n}\int_{M}|\Delta u|^{2}d\mu\notag,
    \end{align}
    where $\lambda_{1}:=\max\{0,-\lambda\}$. For the derivative of $D(t)$, from Lemma \ref{lemma3.1}, we have
    \begin{align}
        D^{\prime}(t)&=h^{\prime}(t)\int_{M}|\nabla u|^{2}d\mu+h(t)\frac{d}{dt}\left(\int_{M}|\nabla u|^{2}d\mu\right)\\
        &=h^{\prime}(t)\int_{M}|\nabla u|^{2}d\mu+h(t)\int_{M}(\partial_{t}-\Delta)|\nabla u|^{2}d\mu\notag\\
        &=h^{\prime}(t)\int_{M}|\nabla u|^{2}d\mu-2h(t)\int_{M}|\nabla^{2}u|^{2}d\mu+2\lambda ph(t)\int_{M}u^{p-1}|\nabla u|^{2}d\mu\notag.
    \end{align}
    If $h(t)< 0$, note that 
    \begin{align}
        D^{\prime}(t)\geq\frac{h^{\prime}(t)+2\lambda_{1} pA^{p-1}h(t)}{h(t)}D(t)-2h(t)\int_{M}|\nabla^{2}u|^{2}d\mu,
    \end{align}
     together with the estimate of $I^{\prime}(t)$ and Lemma \ref{lemma4.1}, yields
    \begin{align}
        I^{2}(t)U^{\prime}(t)&\geq\exp\left\{-\int_{t_0}^{t}{\left(\frac{h^{\prime}(s)+2\lambda_{1} pA^{p-1}h(s)}{h(s)}+\frac{N(s)}{2}n+\frac{ pC(n)}{s}+P\right)ds}\right\}\notag\\
        &\quad\cdot \left[-2h(t)I(t)\int_{M}|\nabla^{2}u|^{2}d\mu+\frac{2h}{N(t)n}\left(\int_{M}|\Delta u|^{2}d\mu\right)\left(\int_{M}|\nabla u|^{2}d\mu\right)\right]\notag\\
        &\geq\exp\left\{-\int_{t_0}^{t}{\left(\frac{h^{\prime}(s)+2\lambda_{1} pA^{p-1}h(s)}{h(s)}+\frac{N(s)}{2}n+\frac{ pC(N)}{s}+P\right)ds}\right\}\notag\\
        &\quad\cdot\left[-\frac{2h}{n}I(t)\left(\int_{M}|\Delta u|^{2}d\mu\right)+\frac{2h}{N(t)n}\cdot N(t)I(t)\left(\int_{M}|\Delta u|^{2}d\mu\right)\right]\notag\\
        &=0\notag.
    \end{align}
    where we let $\displaystyle{N(t)=C_{2}^{2}(t)\left(1+\ln\frac{A}{\eta}\right)}^{2}$.

    On the other hand, if $h(t)>0$, similarly, we have
    \begin{align}
        D^{\prime}(t)\leq\frac{h^{\prime}(t)+2\lambda_{1} pA^{p-1}h(t)}{h(t)}D(t)-2h(t)\int_{M}|\nabla^{2}u|^{2}d\mu,
    \end{align}
    and 
    \begin{align}
        I^{2}(t)U^{\prime}(t)&\leq\exp\left\{-\int_{t_0}^{t}{\left(\frac{h^{\prime}(s)+2\lambda_{1} pA^{p-1}h(s)}{h(s)}+\frac{N(s)}{2}n+\frac{ pC(n)}{s}+P\right)ds}\right\}\notag\\
        &\quad\cdot \left[-2h(t)I(t)\int_{M}|\nabla^{2}u|^{2}d\mu+\frac{2h}{N(t)n}\left(\int_{M}|\Delta u|^{2}d\mu\right)\left(\int_{M}|\nabla u|^{2}d\mu\right)\right]\notag\\
        &=0\notag.
    \end{align}
    Together with the above, we get the desired results.
\end{proof}

If we let 
$$\psi(t)=\exp\left\{-\int_{t_0}^{t}{\left(\frac{h^{\prime}(s)+2\lambda_{1} pA^{p-1}h(s)}{h(s)}+\frac{N(s)}{2}n+\frac{ pC(n)}{s}+P\right)ds}\right\},$$
then we get

\begin{corollary}\label{corollary4.4}
    Suppose that $M$ is a closed $n$-dimensional  Riemannian manifold, $(M,g(t))_{t\in [0,T)}$ is the solution of the Ricci flow $\eqref{RF}$ with $-K_{1}g(t)\leq {\rm Ric}(g(t))\leq K_{2}g(t)$ for some $K_{1},K_{2}>0$, and $u(t)$ is a positive solution of the heat-type equation $\eqref{heat-type equation}$ with $\eta\leq u(t)\leq A$. Then for any $t\in[t_{0},t_{1}]\subset(0,T)$,
    \begin{equation}
        I(t_1)\geq I(t)\exp\left\{2U(t)\int_{t}^{t_{1}}\frac{-1}{\psi(s)h(s)}ds+2\lambda_{1}\eta^{p-1}(t_{1}-t) \right\}.
    \end{equation}
\end{corollary}
\begin{proof}
    We give the proof of case $h(t)<0$ (The case $h(t)>0$ is similar to it). According to the definition of $U(t)$, yields
    \begin{align}
        \frac{d}{dt}\ln I(t)&=\frac{I^{\prime}(t)}{I(t)}=\frac{1}{I(t)}\int_{M}\left[2u\cdot(\partial_{t}-\Delta)u-2|\nabla u|^{2}\right]d\mu\\
        &\geq -\frac{2}{h(t)\psi(t)}U(t)+2\lambda_{1}\eta^{p-1}\notag.
    \end{align}
     Integrate (4.8) from $t$ to $t_{1}$ and by applying Theorem \ref{theorem4.3} to it,  yields
     \begin{align}
         \ln I(t_{1})-\ln I(t)&\geq\int_{t}^{t_{1}}\left(-\frac{2}{h(s)\psi(s)}U(s)+2\lambda_{1}\eta^{p-1}\right)ds\notag\\
         &\geq 2U(t)\int_{t}^{t_{1}}-\frac{ds}{h(s)\psi(s)}+2\lambda_{1}\eta^{p-1}(t_{1}-t)\notag.
     \end{align}
     From the boundedness of time-dependent function $h(t)$, we have
     \begin{align}
         I(t_1)\geq I(t)\exp\left\{2U(t)\int_{t}^{t_{1}}\frac{-1}{\psi(s)h(s)}ds+2\lambda_{1}\eta^{p-1}(t_{1}-t) \right\}\notag.
     \end{align}
     Thus we get the desired result.
\end{proof}

\textbf{Acknowledgments.}\ \ 
The first author and the third author would like to thank Professor Xin-An Ren for useful discussions. The first author is partly supported by the National Key R$\&$D Program of China 2020YFA0712800 and the Interdisciplinary Research Foundation for Doctoral Candidates of Beijing Normal University (Grant BNUXKJC2318). The second author is funded by Shanghai Institute for Mathematics and Interdisciplinary Sciences(SIMIS) under grant number SIMIS-ID-2024-LG.  The authors would also like to thank the referees for their valuable comments and suggestions.

\textbf{ Conflict of Interest}\ \ 
 The authors declare no conflict of interest.

\bibliographystyle{amsplain}

\end{document}